\documentclass[12pt]{article}

\usepackage[colorlinks]{hyperref}            
\usepackage{color}
\usepackage{graphicx,subfigure,amsmath,amssymb,amsfonts,bm,epsfig,epsf,url,dsfont}
\usepackage{times}
\usepackage{bbm}      
\usepackage{booktabs}
\usepackage{cases}
\usepackage{fullpage}
\usepackage[small,bf]{caption}
\usepackage{natbib}
\usepackage[top=1in,bottom=1in,left=1in,right=1in]{geometry}

\renewcommand{\phi}{\varphi}

\newcommand{\mI}{\mathrm{I}}
\newcommand{\mX}{\mathbb{X}}
\newcommand{\mS}{\mathbb{S}}
\renewcommand{\P}{\mathbb{P}}
\newcommand{\E}{\mathbb{E}}
\newcommand{\N}{\mathbb{N}}
\newcommand{\R}{\mathbb{R}}

\newcommand{\ent}{\mathrm{Ent}}

\newcommand{\cG}{\mathcal{G}}
\newcommand{\cW}{\mathcal{W}}

\def\ds1{\mathds{1}}
\renewcommand{\epsilon}{\varepsilon}

\newlength{\minipagewidth}
\newcommand{\beq}{\begin{equation}}
\newcommand{\eeq}{\end{equation}}

\newcommand{\beqa}{\begin{eqnarray}}
\newcommand{\eeqa}{\end{eqnarray}}

\newcommand{\beqan}{\begin{eqnarray*}}
\newcommand{\eeqan}{\end{eqnarray*}}

\def\ba#1\ea{\begin{align*}#1\end{align*}} 
\def\banum#1\eanum{\begin{align}#1\end{align}} 

\newtheorem{theorem}{Theorem}
\newtheorem{lemma}{Lemma}

\newcommand{\BlackBox}{\rule{1.5ex}{1.5ex}}  
\newenvironment{proof}{\par\noindent{\bf Proof\ }}{\hfill\BlackBox\\[2mm]}

\begin{document}

\title{Entropic CLT and phase transition in high-dimensional\\ Wishart matrices}

\author{S\'ebastien Bubeck \thanks{∗Microsoft Research; sebubeck@microsoft.com.}
\and 
Shirshendu Ganguly \thanks{UC Berkeley; sganguly@berkeley.edu. }}
\date{\today}

\maketitle

\begin{abstract}
We consider high dimensional Wishart matrices $\mX \mX^{\top}$ where the entries of $\mX \in \R^{n \times d}$ are i.i.d. from a log-concave distribution. We prove an information theoretic phase transition: such matrices are close in total variation distance to the corresponding Gaussian ensemble if and only if $d$ is much larger than $n^3$. Our proof is entropy-based, making use of the chain rule for relative entropy along with the recursive structure in the definition of the Wishart ensemble. The proof crucially relies on the well known relation between Fisher information and entropy, a variational representation for Fisher information, concentration bounds for the spectral norm of a random matrix, and certain small ball probability estimates for log-concave measures.
\end{abstract}

\section{Introduction}
Let $\mu$ be a probability distribution supported on $\R$ with zero mean and unit variance. We consider a Wishart matrix (with removed diagonal) $W = \left( \mX \mX^{\top} - \mathrm{diag}(\mX \mX^{\top}) \right) / \sqrt{d}$ where $\mX$ is an $n \times d$ random matrix with i.i.d. entries from $\mu$. The distribution of $W$, which we denote $\cW_{n,d}(\mu)$, is of importance in many areas of mathematics. Perhaps most prominently it arises in statistics as the distribution of covariance matrices, and in this case $n$ can be thought of as the number of parameters and $d$ as the sample size. Another application is in the theory of random graphs where the thresholded matrix $A_{i,j} = \ds1\{W_{i,j} >\tau\}$ is the adjacency matrix of a random geometric graph on $n$ vertices, where each vertex is associated to a latent feature vector in $\R^d$ (namely the $i^{th}$ row of $\mX$), and an edge is present between two vertices if the correlation between the underlying features is large enough. Wishart matrices also appear in physics, as a simple model of a random mixed quantum state where $n$ and $d$ are the dimensions of the observable and unobservable states respectively.

The measure $\cW_{n,d}(\mu)$ becomes approximately Gaussian when $d$ goes to infinity and $n$ remains bounded (see Section \ref{sec:mainresult}). Thus in the classical regime of statistics where the sample size is much larger than the number of parameters one can use the well understood theory of Gaussian matrices to study the properties of $\cW_{n,d}(\mu)$.
In this paper we investigate the extent to which this Gaussian picture remains relevant in the {\em high-dimensional regime} where the matrix size $n$ also goes to infinity. Our main result, stated informally, is the following universality of a critical dimension for sufficiently smooth measures $\mu$ (namely log-concave): the Wishart measure  $\cW_{n,d}(\mu)$ becomes approximately Gaussian if and only if $d$ is much larger than $n^3$. From a statistical perspective this means that analyses based on Gaussian approximation of a Wishart are valid as long as the number of samples is at least the cube of the number of parameters. In the random graph setting this gives a dimension barrier to the extraction of geometric information from a network, as our result shows that all geometry is lost when the dimension of the latent feature space is larger than the cube of the number of vertices.

\subsection{Main result} \label{sec:mainresult}
Writing $X_i \in \R^d$ for the $i^{th}$ row of $\mX$ one has for $i \neq j$, $W_{i,j} = \frac{1}{\sqrt{d}} \langle X_i, X_j \rangle$. In particular $\E W _{i,j} = 0$ and $\E W_{i,j} W_{\ell, k} = \ds1\{(i, j) = (\ell, k) \ \text{and} \ i \neq j\}.$ Thus for fixed $n$, by the multivariate central limit theorem one has, as $d$ goes to infinity, 
$$\cW_{n,d}(\mu) \stackrel{D}{\rightarrow} \cG_n ,$$
where $\cG_n$ is the distribution of a $n \times n$ Wigner matrix with null diagonal and standard Gaussian entries off diagonal (recall that a Wigner matrix is symmetric and the entries above the main diagonal are i.i.d.). Recall that the total variation distance between two measures $\lambda, \nu$ is defined as $\mathrm{TV}(\lambda, \nu) = \sup_A |\lambda(A) - \nu(A)|$ where the supremum is over all measurable sets $A$. Our main result is the following:

\begin{theorem} \label{th:mainresult}
Assuming that $\mu$ is log-concave\footnote{ A measure $\mu$  with density $f$ is said to be log-concave if $f(\cdot)=e^{-\phi(\cdot)}$ for some convex function $\phi.$} and $d/ (n^3 \log^2(d)) \rightarrow +\infty$, one has
\begin{equation} \label{eq:mainresult1}
\mathrm{TV}(\cW_{n,d}(\mu), \cG_n) \rightarrow 0.
\end{equation}
\end{theorem}
Observe that for \eqref{eq:mainresult1} to be true one needs some kind of smoothness assumption on $\mu$. Indeed if $\mu$ is purely atomic then so is $\cW_{n,d}(\mu)$, and thus its total variation distance to $\cG_n$ is $1$. We also remark that Theorem \ref{th:mainresult} is tight up to the logarithmic factor in the sense that if $d/ n^3 \rightarrow 0$, then
\begin{equation} \label{eq:mainresult2}
\mathrm{TV}(\cW_{n,d}(\mu), \cG_n) \rightarrow 1 ,
\end{equation}
see Section \ref{sec:related} below for more details on this result. Finally we note that our proof in fact gives the following quantitative version of \eqref{eq:mainresult1}:
\begin{theorem} \label{th:mainresult2} 
There exists a universal constant $C>1$ such that for $d \geq C n^2,$
$$\mathrm{TV}(\cW_{n,d}(\mu), \cG_n)^2 \leq C \left( \frac{n^3 \log^2(d)+n^2\log^4(d)}{d} + \sqrt{\frac{n^3}{d}} \right) .$$
\end{theorem}
\subsection{Related work and ideas of proof} \label{sec:related}
In the case where $\mu$ is a standard Gaussian, Theorem \ref{th:mainresult} (without the logarithmic factor) was recently proven simultaneously and independently in \cite{BDER14, JL13}. We also observe that previously to these results certain properties of a Gaussian Wishart were already known to behave as those of a Gaussian matrix, and for values of $d$ much smaller than $n^3$, see e.g. \cite{Joh01} for the largest eigenvalue at $d \approx n$, and \cite{cpam14} on whether the quantum state represented by the Wishart is separable at $d \approx n^{3/2}$.
The proof of Theorem \ref{th:mainresult} for the Gaussian case is simpler as both measures have a known density with a rather simple form, and one can then explicitly compute the total variation distance as the $L_1$ distance between the densities. \\

We now discuss how to lower bound $\mathrm{TV}(\cW_{n,d}(\mu), \cG_n).$  \cite{BDER14} implicitly proves \eqref{eq:mainresult2} when $\mu$ is Gaussian. Taking inspiration from this, one can show that in the regime $d/ n^3 \rightarrow 0$, for any $\mu$ (zero mean, unit variance and finite fourth moment\footnote{ Note that log-concavity implies exponential tails and hence existence of all moments. See \eqref{subexp1}.}), one can distinguish $\cG_n$ and $\cW_{n,d}(\mu)$ by considering the statistic $A \in \R^{n \times n} \mapsto \mathrm{Tr}(A^3)$. Indeed it turns out that the mean of $\mathrm{Tr}(A^3)$ under the two measures are respectively zero and $\Theta (\frac{n^3}{\sqrt{d}})$ whereas the variances are $\Theta (n^3)$ and $\Theta (n^{3} + \frac{n^5}{{d}^2})$. Since $d=o(n^3)$ implies $\sqrt{n^{3}+\frac{n^5}{{d}^2}}=o(\frac{n^3}{\sqrt{d}}),$  \eqref{eq:mainresult2} follows by a simple application of  Chebyshev's inequality. We omit the details and refer the interested reader to  \cite{BDER14}.\\


Proving normal approximation results without the assumption of independence  is a natural question and has been a subject of intense study over many years. One method that has found several applications in such settings is the so called Stein's method of exchangeable pairs. Since Stein's original work (see \cite{stein1}) the method has been considerably generalized to prove error bounds on convergence to gaussian distribution in various situations. The multidimensional case was treated first in \cite{chatmeck}. For several applications of Stein's method in proving CLT see \cite{chaticm} and the references therein. 
In our setting note that $$W=\sum_{i=1}^{d} \left( \mX_i \mX_i^{\top} - \mathrm{diag}(\mX_i \mX_i^{\top}) \right) / \sqrt{d}$$ where the $\mX_i$ are i.i.d vectors in $\R^n$ whose coordinates are i.i.d samples from a one dimensional measure $\mu.$
Considering $\mathbb{Y}_i=\mX_i \mX_i^{\top} - \mathrm{diag}(\mX_i \mX_i^{\top})$ as a vector in $\R^{n^2}$ and noting that
$|\mathbb{Y}_{i}|^3 \sim n^3,$ a straightforward application of Stein's method using exchangeable pairs (see the proof of \cite[Theorem 7]{chatmeck}) provides the following suboptimal bound: the Wishart ensemble converges to the Gaussian ensemble (convergence of integrals against `smooth' enough test functions) when $d \gg n^6.$ Whether there is a way to use Stein's method to recover Theorem \ref{th:mainresult} in any reasonable metric (total variation metric, Wasserstein metric, etc.) remains an open problem (see Section \ref{sec:open} for more on this).
\newline

Our approach to proving \eqref{eq:mainresult1} is information theoretic and hence completely different from \cite{BDER14, JL13} (this is a necessity since for a general $\mu$ there is no simple expression for the density of $\cW_{n,d}(\mu)$). The first step in our proof, described in Section \ref{sec:induction}, is to use Pinsker's inequality to change the focus from total variation distance to the relative entropy (see also Section \ref{sec:induction} for definitions). Together with the chain rule for relative entropy this allows us to bound the relative entropy of $\cW_{n,d}(\mu)$ with respect to $\cG_n$ by induction on the dimension $n$. The base case essentially follows from the work of \cite{ABBN04} who proved that the relative entropy between the standard one-dimensional Gaussian and $\frac{1}{\sqrt{d}} \sum_{i=1}^d x_i$, where $x_1, \hdots, x_d \in \R$ is an i.i.d. sequence from a log-concave measure $\mu$, goes to $0$ at a rate $1/d$. One of the main technical contribution of our work is a certain generalization of the latter result in higher dimensions, see Theorem \ref{th:highdimCLT} in Section \ref{sec:highdimentropicCLT}. Recently \cite{BN12} also studied a high dimensional generalization of the result in \cite{BBN03} (which contains the key elements for the proof in \cite{ABBN04}) but it seems that Theorem \ref{th:highdimCLT} is not comparable to the main theorem in \cite{BN12}.

Another important part of the induction argument, which is carried out in Section \ref{sec:smallball}, relies on controlling from above the expectation of $-\mathrm{logdet}(\frac{1}{d} \mX \mX^{\top})$, which should be understood as the relative entropy between a centered Gaussian with covariance given by $\frac{1}{d} \mX \mX^{\top}$ and a standard Gaussian in $\R^n$. This leads us to study the probability that $\mX \mX^{\top}$ is close to being non-invertible.  Denoting by $s_{\mathrm{min}}$ the smallest singular value of $\mX$, it suffices to prove a `good enough' upper bound for $\P(s_{\mathrm{min}}(\mX^{\top}) \leq \epsilon)$ for all small $\epsilon$.
The case when the entries of $\mX$ are gaussian allows to work with exact formulas and was studied in \cite{ede88,sankar06}. The last few years have seen tremendous progress in understanding the universality of the tail behavior of extreme singular values of random matrices with i.i.d. entries from general distributions. See \cite{RV10} and the references therein for a detailed account of these results. 
Such estimates are quite delicate, and it is worthwhile to mention that the following estimate was proved only recently in \cite{RV08}: Let $A \in \R^{n \times d}$ with $(d\ge n)$ be a rectangular matrix with i.i.d. subgaussian entries then for all $\epsilon >0,$  $$\P(s_{\mathrm{min}}(A^{\top}) \leq \epsilon(\sqrt{d}-\sqrt{n-1})) \leq (C\epsilon)^{d-n+1} +c^d,$$ where $c,C$ are independent of $n,d$. In full generality, such estimates are essentially sharp since in the case where the entries are random signs, $s_{\mathrm{min}}$ is zero with probability $c^d$. Unfortunately this type of bound is not useful for us, as we need to control $\P(s_{\mathrm{min}}(\mX^{\top}) \leq \epsilon)$ for {\em{arbitrarily small scales}}  $\epsilon$ (indeed $\mathrm{logdet}(\frac{1}{d} \mX \mX^{\top})$ would blow up if $s_{\mathrm{min}}$ can be zero with non-zero probability). It turns out that the assumption of log-concavity of the distribution allows us to do that. To this end we use recent advances in \cite{Pao12} on small ball probability estimates for such distributions: Let $Y \in \R^n$ be an isotropic centered log-concave random variable, and $\epsilon \in (0,1/10)$, then one has $\P(|Y| \leq \epsilon \sqrt{n}) \leq (C \epsilon)^{\sqrt{n}}$. This together with an $\epsilon$-net argument gives us the required control on $\P(s_{\mathrm{min}}(\mX^{\top}) \leq \epsilon)$.
\newline

We conclude the paper with several open problems in Section \ref{sec:open}.

\section{An induction proof via the chain rule for relative entropy} \label{sec:induction}
Recall that the (differential) entropy of a measure $\lambda$ with a density $f$ (all densities are understood with respect to the Lebesgue measure unless stated otherwise) is defined as:
$$\ent(\lambda) = \ent(f) = - \int f(x) \log f(x) dx .$$
The relative entropy of a measure $\lambda$ (with density $f$) with respect to a measure $\nu$ (with density $g$) is defined as
$$\ent(\lambda \Vert \nu) = \int f(x) \log \frac{f(x)}{g(x)} dx .$$
With a slight abuse of notations we sometimes write $\ent(Y \Vert \nu)$ where $Y$ is a random variable distributed according to some distribution $\lambda$. Pinsker's inequality gives:
$$\mathrm{TV}(\cW_{n,d}(\mu), \cG_n)^2 \leq \frac{1}{2} \ent(\cW_{n,d}(\mu) \Vert \cG_n) .$$
Next recall the chain rule for relative entropy states for any random variables $Y_1, Y_2, Z_1, Z_2$,
$$\ent( (Y_1, Y_2) \Vert (Z_1, Z_2) ) = \ent(Y_1 \Vert Z_1) + \E_{y \sim \lambda_1} \ent( Y_2 \vert Y_1=y \Vert Z_2 \vert Z_1=y ) ,$$
where $\lambda_1$ is the (marginal) distribution of $Y_1$, and $Y_2 \vert Y_1=y$ is used to denote the distribution of $Y_2$ conditionally on the event $Y_1 = y$ (and similarly for $Z_2 \vert Z_1=y$). Also observe that a sample from $\cW_{n+1,d}(\mu)$ can be obtained by adjoining to $\left( \mX \mX^{\top} - \mathrm{diag}(\mX \mX^{\top}) \right) / \sqrt{d}$ (whose distribution is $\cW_{n,d}(\mu)$) the column vector $\mX X /\sqrt{d}$ (and the row vector $(\mX X)^{\top} /\sqrt{d}$) where $X \in \R^d$ has i.i.d. entries from $\mu$. Thus denoting $\gamma_n$ for the standard Gaussian measure in $\R^n$ we obtain for all $n\ge 1,$
\begin{equation} \label{eq:induction1}
\ent(\cW_{n+1,d}(\mu) \Vert \cG_{n+1}) = \ent(\cW_{n,d}(\mu) \Vert \cG_{n}) + \E_{\mX} \ \ent\left(\mX X /\sqrt{d} \ \vert \ \mX \mX^{\top} \Vert \gamma_{n}\right) .
\end{equation}
By convexity of the relative entropy (see e.g., \cite{CT91}) one also has:
\begin{equation} \label{eq:induction2}
\E_{\mX} \ \ent(\mX X /\sqrt{d} \ \vert \ \mX \mX^{\top} \ \Vert \gamma_{n}) \leq \E_{\mX} \ \ent(\mX X  /\sqrt{d} \ \vert \ \mX \ \Vert \gamma_{n}) .
\end{equation}
Also, since by definition both $\cW_{1,d}(\mu)$ and  $\cG_{1}$ are zero, $\ent(\cW_{1,d}(\mu) \Vert \cG_{1})=0$ as well.

Next we need a simple lemma to rewrite the right hand side of \eqref{eq:induction2}:
\begin{lemma} \label{lem:rewritingentropy}
Let $A \in \R^{n \times d}$ and $Q \in \R^{n \times n}$ be such that $Q A A^{\top} Q^{\top} = \mI_n$. Then one has for any isotropic random variable $X \in \R^d$,
$$\ent(A X \Vert \gamma_{n}) = \ent(Q A X \Vert \gamma_{n}) + \frac{1}{2} \mathrm{Tr}(A A^{\top}) - \frac{n}{2} +  \mathrm{logdet}(Q) .$$
\end{lemma}

\begin{proof}
Denote $\Phi_{\Sigma}$ for the density of a centered $\R^n$ valued,  Gaussian with covariance matrix $\Sigma$ (i.e., $\Phi_{\Sigma}(x) = \frac{1}{\sqrt{(2 \pi)^n \mathrm{det}(\Sigma)}} \exp(- \frac{1}{2} x^{\top} \Sigma^{-1} x )$), and let $G \sim \gamma_n$. Also let $f$ be the density of $Q A X$. Then one has (the first equality is a simple change of variables):
\begin{eqnarray*}
\ent(A X \Vert G) & = & \ent(Q A X \Vert Q G) \\ 
& = & \int f(x) \log \left( \frac{f(x)}{\Phi_{Q Q^{\top}}(x)} \right) dx \\
& = & \int f(x) \log \left( \frac{f(x)}{\Phi_{\mI_n}(x)} \right) dx + \int f(x) \log \left( \frac{\Phi_{\mI_n}(x)}{\Phi_{Q Q^{\top}}(x)} \right) dx\\
& = & \ent(Q A X \Vert G) +  \int f(x) \left(\frac{1}{2} x^{\top} (Q Q^{\top})^{-1} x - \frac{1}{2} x^{\top} x + \frac{1}{2} \mathrm{logdet}(QQ^{\top}) \right) \\
& = & \ent(Q A X \Vert G) + \frac{1}{2} \mathrm{Tr}\left( (Q Q^{\top})^{-1} \right) - \frac{n}{2} +  \mathrm{logdet}(Q) ,
\end{eqnarray*}
where for the last equality we used the fact that $Q A X$ is isotropic, that is $\int f(x) x x^{\top} dx = \mI_n$ and $\mathrm{det}(QQ^T)=\mathrm{det}(Q)^2$. Finally it only remains to observe that $\mathrm{Tr}\left( (Q Q^{\top})^{-1} \right) = \mathrm{Tr}(A A^{\top})$.
\end{proof}
Combining \eqref{eq:induction1} and \eqref{eq:induction2} with Lemma \ref{lem:rewritingentropy} (noting that one can take $Q = (\frac{1}{d} \mX \mX^{\top})^{-1/2}$), and using that $\E \ \mathrm{Tr}(\mX \mX^{\top}) = nd$, one obtains
\begin{align}
& \ent(\cW_{n+1,d}(\mu) \Vert \cG_{n+1}) \notag \\
& \leq \ent(\cW_{n,d}(\mu) \Vert \cG_{n}) + \E_{\mX} \ \ent\left((\mX \mX^{\top})^{-1/2} \mX \ X \ \vert \ \mX \ \Vert \gamma_{n}\right) - \frac{1}{2} \E_{\mX} \ \mathrm{logdet} (\frac{1}{d} \mX \mX^{\top}) . \label{eq:lefttodo}
\end{align}
In Section \ref{sec:highdimentropicCLT} we show how to bound the term $\ent(A X \Vert \gamma_n)$ where $A \in \R^{n \times d}$ has orthonormal rows (i.e., $A A^{\top} = \mI_n$) and thereby proving a central limit theorem. In Section \ref{sec:smallball} we deal with the term $\E_{\mX} \ \mathrm{logdet} (\frac{1}{d} \mX \mX^{\top})$.
The proof of Theorem \ref{th:mainresult2} and hence Theorem \ref{th:mainresult} would thus follow by iterating \eqref{eq:lefttodo} and the results of these sections. 

\section{A high dimensional entropic CLT} \label{sec:highdimentropicCLT}
The main goal of this section is to prove the following high dimensional generalization of the entropic CLT of \cite{ABBN04}.

\begin{theorem} \label{th:highdimCLT}
Let $Y \in \R^d$ be a random vector with i.i.d. entries from a distribution $\nu$ with zero mean, unit variance, and spectral gap\footnote{A probability measure $\mu$ is said to have spectral gap $c$ if for all smooth functions $g$ with $\E_{\mu}(g)=0,$  we have $\E_{\mu}(g^2) \le \frac{1}{c}\E_{\mu}(g'^2).$
} $c\in (0,1]$.
Let $A \in \R^{n \times d}$ be a matrix such that $A A^{\top} = \mI_n$. Let $\epsilon = \max_{i\in [d]} (A^{\top} A)_{i,i}$ and $\zeta = \max_{i,j \in [d], i \neq j} |(A^{\top} A)_{i,j}|$. Then one has,
$$\ent(A Y \Vert \gamma_n) \leq n \min(2 (\epsilon + \zeta^2 d) / c, 1) \ \ent(\nu \Vert \gamma_1) .$$ 
\end{theorem}

Note that the assumption $A A^{\top} = \mI_n$ implies that the rows of $A$ form an orthonormal system. In particular if $A$ is built by picking rows one after the other at uniform on the Euclidean sphere in $\R^d$ conditionally on being orthogonal to previous rows, then one expects that $\epsilon \simeq n / d$ and $\zeta \simeq \sqrt{n} / d$. Theorem \ref{th:highdimCLT} then yields $\ent( A Y \Vert \gamma_n) \lesssim n^2 / d$. Thus we already see appearing the term $n^3 / d$ from Theorem \ref{th:mainresult} as we will sum the latter bound over the $n$ rounds of induction (see Section \ref{sec:induction}).

We also note that for the special case $n=1$, Theorem \ref{th:highdimCLT} is slightly weaker than the result of \cite{ABBN04} which makes appear the $\ell_4$-norm of $A$.

Section \ref{sec:highdimCLTproof1} and Section \ref{sec:highdimCLTproof2} are dedicated to the proof of Theorem \ref{th:highdimCLT}. Then in Section \ref{sec:using} we show how to apply this result to bound the term $\E_{\mX} \ \ent(Q \mX X / \sqrt{d} \ \vert \ \mX \Vert \gamma_{n})$ from Section \ref{sec:induction}.

\subsection{From entropy to Fisher information} \label{sec:highdimCLTproof1}
For a density function $w : \R^n \rightarrow \R_+$, let $J(w) := \int_{\R^n} \frac{|\nabla w(x)|^2}{w(x)} dx$ denote its Fisher information (where $\nabla w(\cdot)$ denotes the gradient vector of $w$ and $|\cdot|$ denotes the euclidean norm), and $I(w) := \int \frac{\nabla w(x) \nabla w(x)^{\top}}{w(x)} dx,$  the Fisher information matrix (if $\nu$ denotes the measure whose density is $w$, we may also write $J(\nu)$ instead of $J(w)$). We use  $P_t$ to denote the Ornstein-Uhlenbeck semigroup, i.e., for a random variable $Z$ with density $g$, we define
$$P_t Z := \exp(-t) Z + \sqrt{1 - \exp(-2 t)} G ,$$ where $G \sim \gamma_n$ (the standard Gaussian in $\R^n$) is independent of $Z$; we denote by $P_t g$, the density of $P_t Z$. The de Bruijn identity states that the Fisher information is the time derivative of the entropy along the Ornstein-Uhlenbeck semigroup, more precisely one has for any centered and isotropic density $w$ :
$$\ent(w \Vert \gamma_n) = \ent(\gamma_n) - \ent(w) = \int_{0}^{\infty} (J(P_t w) - n) dt,$$
(the first equality is a simple consequence of the form of the normal density).
Our objective is to prove a bound of the form (for some constant $C$ depending on $A$)
\begin{equation} \label{eq:toprove1}
\ent(A Y \Vert \gamma_n) \leq C \ \ent(\nu \Vert \gamma_1),
\end{equation}
and thus given the above identity it suffices to show that for any $t > 0$,
\begin{equation} \label{eq:toprove2}
J(h_t) - n \leq C \ (J(\nu_t) - 1),
\end{equation}
where $h_t$ is the density of $P_t A Y$ (which is equal to the density of $A P_t Y$) and $\nu_t$ is such that $P_t Y$ has distribution $\nu_t^{\otimes d}$.
Furthermore if $e_1, \hdots, e_n$ denotes the canonical basis of $\R^n$, then to prove \eqref{eq:toprove2} it is enough to show that for any $i \in [n]$,
\begin{equation} \label{eq:toprove3}
e_i^{\top} I(h_t) e_i - 1 \leq C_i \ (J(\nu_t) - 1) ,
\end{equation}
where $\sum_{i=1}^n C_i = C$. Recall $c$ is the spectral gap of $\nu.$  We will show that one can take,
$$C_i = 1 - \frac{c U_i^2}{c W_i + 2 V_i} ,$$
where we denote $B= A^{\top} A \in \R^{d \times d}$, and
$$U_i =\sum_{j=1}^d A^2_{i,j}(1-B_{j,j}) , \, 
W_i =\sum_{j=1}^d A^2_{i,j}(1-B_{j,j})^2 , \,
V_i = \sum_{j,k \in [d], k \neq j} (A_{i,j}B_{j,k})^2.$$
Straightforward calculations (using that $U_i \geq 1- \epsilon$, $W_i \leq 1$, and $V_i \leq \zeta^2 d$) show that one has $\sum_{i=1}^n \left( 1 - \frac{c U_i^2}{c W_i + 2 V_i}  \right) \leq 2 n (\epsilon + \zeta^2 d) / c$ where $\epsilon = \max_{i\in [d]} B_{i,i}$ and $\zeta = \max_{i,j \in [d], i \neq j} |B_{i,j}|$, thus concluding the proof of Theorem \ref{th:highdimCLT}.

In the next subsection we prove \eqref{eq:toprove3} for a given $t>0$ and $i=1$. We use the following well known but crucial fact: the spectral gap of $\nu_t$ is in $[c,1]$ (see [Proposition 1, \cite{BBN03}]).

Denoting $f$ for the density of $\nu_t$, one has with $\phi = - \log f$ that $J:=J(\nu_t) = \int \phi''(x) d\mu(x)$. 
The last equality easily follows from the fact that for any $t > 0$ one has $\int f'' =0$ (which itself follows from the smoothness of $\nu_t$ induced by the convolution of $\nu$ with a Gaussian).

\subsection{Variational representation of Fisher information} \label{sec:highdimCLTproof2}
Let $Z \in \R^d$ be a random variable with a twice continuously differentiable density $w$ such that $\int \frac{|\nabla w|^2}{w} < \infty$ and $\int \Vert \nabla^2 w \Vert < \infty$, and let $h$ the density of $A Z \in \R^n$. Our main tool is a remarkable formula from \cite{BBN03}, which states the following:  for all $e\in \R^n$ and all sufficiently smooth map $p : \R^d \rightarrow \R^d$ with $A p(x) = e, \forall x \in \R^d$, one has (with $D p$ denoting the Jacobian matrix of $p$),
\begin{equation} \label{eq:remarkable}
e^{\top} I(h) e \leq \int \bigg( \mathrm{Tr}(Dp(x)^2) + p(x)^{\top} \nabla^2 (- \log w(x)) p(x) \bigg) w(x) dx .
\end{equation}
For sake of completeness we include a short proof of this inequality in Section \ref{finalsec}.

Let $(a_{1}, \hdots, a_{d})$ be the first row of $A$. Following \cite{ABBN04}, to prove \eqref{eq:toprove2}, we would like to use the above formula\footnote{Note that the smoothness assumptions on $w$ are satisfied in our context since we consider a random variable convolved with a Gaussian.} with $p$ of the form $(a_{1} r(x_1), \hdots, a_{d} r(x_d))$ for some map $r : \R \rightarrow \R$. Since we need to satisfy $A p(x) = e_1$ we adjust the formula accordingly and take
$$p(x) = (\mI_d - A^{\top} A) (a_{1} r(x_1), \hdots, a_{d} r(x_d))^{\top} + A^{\top} e_1 .$$
In particular we get, with $B= A^{\top} A$,
$$p_i(x) = a_i+a_i(1-B_{i,i})  r(x_i)-\sum_{j \in [d], j\neq i} B_{i,j}a_j r(x_j) ,$$
and
$$\frac{\partial p_i}{\partial x_j}(x)=\left \{ 
\begin{array}{cc}
a_i(1 -B_{i,i})r'(x_i) & \mbox{if } i = j\\
-B_{i,j}a_{j}r'(x_j) & \mbox{otherwise}  .
\end{array}\right.
$$
Next recall that we apply \eqref{eq:remarkable} to prove \eqref{eq:toprove3} where $w(x) = \prod_{i=1}^d f(x_i)$, in which case we have (recall also the notation $\phi = - \log f$):
\begin{align*}
p(x)^{\top} \nabla^2 (- \log w(x)) p(x)  & = \sum_{i =1}^d p_i(x)^2 \phi''(x_i) \\
& = \sum_{i =1}^d \phi''(x_i) \left( a_i+a_i(1-B_{i,i})  r(x_i)-\sum_{j \in [d], j\neq i} B_{i,j}a_j r(x_j) \right)^2.
\end{align*}
We also have
$$\mathrm{Tr}(Dp(x)^2) = \sum_{i =1}^d a_i ^2(1 -B_{i,i}) ^2 r'(x_i)^2+\sum_{i,j \in [d], i \neq j} B^2_{i,j}a_{i}a_{j}r'(x_i)r'(x_j) .$$
Putting the above together we obtain (with a slightly lengthy straightforward computation) that $e_1^{\top} I(h) e_1$ is upper bounded by (recall also that $\sum_i a_i^2 =1$ and $\sum_{j} B_{i,j} a_j = a_i$ since $B A^{\top} = A^{\top}$)
\begin{align}\label{eq:expand}
J+W\left(\int f(r')^2 + \int f \phi'' r^2\right)+JV \int f r^2+ J(W-V)\left(\int f r\right)^2 \\
\nonumber
+2U\left(\int f \phi'' r -J\int fr \right) -2W\left(\int f r \right)\left( \int f \phi'' r\right) +M \left( \int fr'\right)^2
\end{align}
where 
$$U= \sum_{i=1}^d a^2_i(1-B_{ii}), \, W=\sum_{i=1}^d a^2_i(1-B_{ii})^2 , \,V= \sum_{i,j \in [d], i \neq j} (B_{i,j} a_j)^2, \, M= \sum_{i,j \in [d], i \neq j} B^2_{i,j} a_i a_j.$$
Observe that by Cauchy-Schwarz inequality one has $M \le V$, and furthermore following \cite{ABBN04} one also has with $m=\int f r$,
 \begin{align*}
\left(\int f r'\right)^2= \left(\int f' (r-m)\right)^2 =\left(\int \frac{f'}{\sqrt f} \sqrt{f} (r-m)\right)^2 \le  J \left(\int f r^2-m^2\right) .
\end{align*}
Thus we get fom \eqref{eq:expand} and the above observations that $e_1^{\top} I(h_t) e_1 - J \le T(r)$ where 
\begin{align*}
T(r)&=W\left ( \int f(r')^2 + \int f \phi'' r^2\right) + 2JV \left( \int f r^2\right)+J(W-2V)\left( \int fr\right)^2 \\
&+2U\left(f\phi''r-J\int fr\right)-2W\left(\int f r\right)\left(\int f \phi'' r\right) ,
\end{align*}
which is the exact same quantity as the one obtained in \cite{ABBN04}. The goal now is to optimize over $r$ to make this quantity as negative as possible. Solving the above optimization problem is exactly the content of \cite[Section 2.4]{ABBN04} and it yields the following bound:
$$e_1^{\top} I(h_t) e_1 - 1 \le \left[ 1-\frac{c U^2}{c W+2V}\right] (J-1) ,$$
which is exactly the claimed bound in \eqref{eq:toprove3}.

\subsection{Using Theorem \ref{th:highdimCLT}} \label{sec:using}
Throughout this section we will assume $d\ge n,$ to have cleaner expressions for some of the error bounds.
Given \eqref{eq:lefttodo} we want to apply Theorem \ref{th:highdimCLT} with $A = (\mX \mX^{\top})^{-1/2} \mX$ (also observe that the spectral gap assumption of Theorem \ref{th:highdimCLT} is satisfied since log-concavity and isotropy of $\mu$ impy that $\mu$ has a spectral gap in $[1/12,1]$, \cite{Bob99}). In particular we have $A^{\top} A = \mX^{\top}  (\mX \mX^{\top})^{-1} \mX$, and thus denoting $\mX_i \in \R^n$ for the $i^{th}$ column of $\mX$ one has for any $i, j \in [d]$,
$$(A^{\top} A)_{i,j} = \mX_i^{\top}  (\mX \mX^{\top})^{-1} \mX_j = \frac{1}{d} \mX_i^{\top}  \big(\frac{1}{d} \mX \mX^{\top}\big)^{-1} \mX_j.$$
In particular this yields:
$$|(A^{\top} A)_{i,j}| \leq \frac{1}{d} |\mX_i^{\top} \mX_j| + \frac{1}{d} | \mX_i | \cdot |\mX_j| \cdot \Vert \big(\frac{1}{d} \mX \mX^{\top}\big)^{-1} - \mI_n \Vert ,$$
where $||\cdot||$ denotes the operator norm.
We now recall two important results on log-concave random vectors\footnote{We note that more classical inequalities could also be used here since the entries of $\mX$ are independent. This would slightly improve the logarithmic factors but it would obscure the main message of this section so we decided to use the more general inequalities for log-concave vectors.}. First Paouris' inequality (\cite{Pao06}  \cite[Theorem 2]{guedon}) states that for an isotropic, centered, log-concave random variable $Y \in \R^n$ one has for any $t \geq C$,
\begin{equation}\label{subexp1}
\P(|Y| \geq (1+t) \sqrt{n}) \leq \exp(- c t \sqrt{n}),
\end{equation}
where $c, C$ are universal constants. We also need an inequality proved by Adamczak, Litvak, Pajor and Tomczak-Jaegermann \cite[Theorem 4.1]{ALPT10} which states that for a sequence $Y_1, \hdots, Y_d \in \R^n$ of i.i.d. copies of $Y$, one has for any $t \geq 1$ and $\epsilon \in (0,1)$,
\begin{equation}\label{normbound}
\P\left( \left\Vert \frac{1}{d} \sum_{i=1}^d Y_i Y_i^{\top} - \mI_n \right\Vert > \epsilon \right) \leq \exp( - c t \sqrt{n} ),
\end{equation}
provided that $d \geq C \frac{t^4}{\epsilon^2} \log^2\left(2 \frac{t^2}{\epsilon^2} \right) n$.
Paouris' inequality \eqref{subexp1} directly yields that for any $i \in [d]$, with probability at least $1-\delta$, one has 
$$| \mX_i | \leq \sqrt{n} + \frac{1}{c} \log(1/\delta) .$$
Furthermore, by a well known consequence of  Pr\'ekopa-Leindler's inequality, conditionally on $\mX_j$ one has for $i \neq j$ that $\mX_i^{\top} \frac{\mX_j}{|\mX_j|}$ is a centered, isotropic, log-concave random variable. In particular using  \eqref{subexp1} and independence of $\mX_i$ and $\mX_j$ one obtains that for $i \neq j$, with probability at least $1-\delta$, 
$$|\mX_i^{\top} \mX_j| \leq |\mX_j| \left(1 + \frac1{c} \log(1/\delta)\right) .$$
To use \eqref{normbound} we plug in $t=C \log(1/\delta)/\sqrt{n}$ for a suitable constant $C$ so that $\exp( - c t \sqrt{n} )$ is at most $\delta.$
Thus $\epsilon$ needs to be such that $d \geq C \frac{t^4}{\epsilon^2} \log^2\left(2 \frac{t^2}{\epsilon^2} \right) n.$ Also without loss of generality by possibly choosing the value of $t$ to be a constant times larger, we can assume $\frac{d}{t^2n}$ lies outside a fixed interval containing $1.$
A suitable value of $\epsilon$  can now be seen from the following string of inequalities, in which we use the fact that $x\log^2x$ is increasing outside a neighborhood of $1$ (the value of the constant $C$ will change from line to line):
$$
\begin{array}{cccc}
&\frac{d}{t^2n} &\ge&   C \frac{t^2}{\epsilon^2}\log^2(2\frac{t^2}{\epsilon^2}), \\\\
\text{if} & \frac{\epsilon^2}{t^2}  &\ge & C \log^2(\frac{d}{t^2n})\frac{t^2n}{d}, \\\\
\text{if} & \epsilon  &\ge & C\frac{t^2\sqrt{n} }{\sqrt{d}}|\log (\frac{d}{t^2n})|.
\end{array}
$$
%
Plugging in our choice of $t=C \log(1/\delta)/\sqrt{n},$ we see that any, $$\epsilon \ge C \sqrt{\frac{1}{dn}} \log^2(1/\delta)[\log(d)+\log\log(1/\delta)],$$ works.
Thus with probability at least $1-\delta$, 


\begin{equation} \label{eq:adamczak}
\Vert \frac{1}{d} \mX \mX^{\top} - \mI_n \Vert \leq C' \sqrt{\frac{1}{dn}} \log^2(1/\delta)[\log(d)+\log\log(1/\delta)].
\end{equation} 

Also note that if $\Vert A - \mI_n \Vert \leq \epsilon < 1$ then $\Vert A^{-1} - \mI_n \Vert \leq \frac{\epsilon}{1-\epsilon}$. From now on $C$ denotes a universal constant whose value can change at each occurence. Putting together all of the above with a union bound, we obtain for $d \geq C n^2$ that with probability at least $1-1/d$, simultaneously for all $i \neq j$,
\begin{align*}
& |\mX_i| \leq C (\sqrt{n} + \log(d)), \\
& |\mX_i^{\top} \mX_j| \leq C (\sqrt{n} \log(d) + \log^2(d)), \\
& \Vert \big(\frac{1}{d} \mX \mX^{\top}\big)^{-1} - \mI_n \Vert  \leq C {\frac{1}{\sqrt{n}}}, 
\end{align*}
where the last inequality follows from \eqref{eq:adamczak} by plugging in $\delta=1/d$ and using the fact that $\log^4(d)=o(\sqrt{d}).$
This yields (using the bounds in the previous page) that with probability at least $1-\frac{1}{d}$ simultaneously for all $i \neq j$,
$$|(A^{\top} A)_{i,j}| \leq C \frac{\sqrt{n} \log(d) + \log^2(d)}{d} ,$$
and
$$|(A^{\top} A)_{i,i}| \leq C \frac{n + \log^2(d)}{d} .$$
Thus denoting $\epsilon = \max_{i\in [d]} (A^{\top} A)_{i,i}$ and $\zeta = \max_{i,j \in [d], i \neq j} |(A^{\top} A)_{i,j}|$ one has:
$$\E \min(\epsilon + \zeta^2 d, 1) \leq  C \frac{n \log^2(d) + \log^4(d)}{d}.$$ 
By  Theorem \ref{th:highdimCLT}, this bounds one of the terms in the upper bound in \eqref{eq:lefttodo}. Thus to complete the proof of Theorem \ref{th:mainresult2} all that is left to do is bound the term $\E_{\mX}[ \ -\mathrm{logdet} (\frac{1}{d} \mX \mX^{\top})]$.
That is the goal of the next section.
 
\section{Small ball probability estimates} \label{sec:smallball}

\begin{lemma}There exists universal $C>0$ such that for $d \geq C n^2,$
\begin{equation} \label{eq:finalsmallball} \E \left( -\mathrm{logdet} \big(\frac{1}{d} \mX \mX^{\top}\big) \right)  \leq C \left(\sqrt{\frac{n}{d}} + \frac{n^2}{d} \right).
\end{equation}
\end{lemma}

\begin{proof}
 We decompose this expectation on the event (and its complement) that the smallest eigenvalue $\lambda_{\mathrm{min}}$ of $\frac{1}{d} \mX \mX^{\top}$ is less than $1/2$. We first write, using $-\log(x) \leq 1-x + 2(1-x)^2$ for $x \geq 1/2$,
$$\E \left( -\mathrm{logdet} \big(\frac{1}{d} \mX \mX^{\top}\big) \ds1\{\lambda_{\mathrm{min}} \geq 1/2\} \right) \leq \E \left( \left|\mathrm{Tr} \big(\mI_n - \frac{1}{d} \mX \mX^{\top}\big) \right| + 2\left\Vert \mI_n - \frac{1}{d} \mX \mX^{\top} \right\Vert_{\mathrm{HS}}^2 \right) ,$$
where $||\cdot||_{\mathrm{HS}}$ denotes the Hilbert-Schmidt norm.
Denote $\zeta$ for the $4^{th}$ moment of $\mu$. Then one has,(recall that $X_i \in \R^d$ denotes the $i^{th}$ row of $\mX$),
$$\E \ \left|\mathrm{Tr} \big(\mI_n - \frac{1}{d} \mX \mX^{\top}\big) \right| \leq \sqrt{\E \left(\mathrm{Tr} \big(\mI_n - \frac{1}{d} \mX \mX^{\top}\big) \right)^2} = \sqrt{\E \left( \sum_{i=1}^n (1 - |X_i|^2 / d) \right)^2} = \sqrt{(\zeta-1) \frac{n}{d}} .$$
Similarly one can easily check that,
$$\E \ \left\Vert \mI_n - \frac{1}{d} \mX \mX^{\top} \right\Vert_{\mathrm{HS}}^2 = \left(\sum_{i,j=1}^n \frac{1}{d^2} \E \ \langle X_i, X_j \rangle^2\right) - n = \frac{n^2 - n}{d}+ \frac{n}{d}(\zeta - 1) \leq \frac{n^2}{d} .$$
Next note that by log-concavity of $\mu$ one has $\zeta \leq 70$, and thus we proved (for some universal constant $C>0$):
\begin{equation} \label{eq:firstpart}
\E \left( -\mathrm{logdet} \big(\frac{1}{d} \mX \mX^{\top}\big) \ds1\{\lambda_{\mathrm{min}} \geq 1/2\} \right)  \leq C \left(\sqrt{\frac{n}{d}} + \frac{n^2}{d} \right).
\end{equation}
We now take care of the integral on the event $\{\lambda_{\mathrm{min}} < 1/2\}$. First observe that for a large enough constant $C>0,$ \eqref{eq:adamczak}  gives for $d \geq C$,
$\P(\lambda_{\mathrm{min}} < 1/2)\leq \exp(- d^{1/10}).$
In particular we have for any $\xi \in (0,1)$:
\begin{align}
\E \left( -\mathrm{logdet} \big(\frac{1}{d} \mX \mX^{\top}\big) \ds1\{\lambda_{\mathrm{min}} < 1/2\}\right) & \leq n \E \left( -\log(\lambda_{\mathrm{min}}) \ds1\{\lambda_{\mathrm{min}} < 1/2\} \right) \notag \\
& = n \int_{\log(2)}^{\infty} \P(-\log(\lambda_{\mathrm{min}}) \geq t) dt \notag \\
& = n \int_{0}^{1/2} \frac{1}{s} \P(\lambda_{\mathrm{min}} < s) ds \notag \\
& \leq \frac{n}{\xi} \exp(-d^{1/10}) + n \int_{0}^{\xi} \frac{1}{s} \P(\lambda_{\mathrm{min}} < s) ds . \label{eq:integralsmallball}
\end{align}
We will choose $\xi$ to be a suitable power of $d$ and the proof will be complete once we control $\P(\lambda_{\mathrm{min}} < s)$ for 
$s\le \xi$. This essentially boils down to estimation of certain small ball probabilities. 
We proceed by bounding the maximum eigenvalue, $\lambda_{\mathrm{max}},$ using a standard net argument.   Note that for any $\epsilon$- net $\mathcal{N}_{\epsilon}$ on $S^{n-1},$

$$\lambda_{\mathrm{max}}= \sup_{\theta \in \mS^{n-1}}  \theta^{\top} \frac{\mX \mX^{\top}}{d} \theta \le \frac{1}{(1-\epsilon)^2} \sup_{\theta \in \mathcal{N}_{\epsilon}}  \theta^{\top} \frac{\mX \mX^{\top}}{d} \theta.$$
Choosing $\epsilon=1/2$ gives $|\mathcal{N}_{\epsilon}| \le 5^n.$
Putting everything together along with subexponential tail of isotropic log-concave random variables (see \eqref{subexp1}) we get, 
%
$\P(\lambda_{\mathrm{max}} > M) \le 5^n \exp(-c \sqrt{M d}),$
(for more details see \cite{RV10}).
Similarly observe,
$$\P(\lambda_{\mathrm{min}} < s) = \P\left( \exists \theta \in \mS^{n-1} : \theta^{\top} \frac{\mX \mX^{\top}}{d} \theta < s \right)  = \P\left( \exists \theta \in \mS^{n-1} : |\mX^{\top} \theta| < \sqrt{s d} \right) .$$
Furthermore, if $|\frac{1}{\sqrt{d}} \mX^{\top} \theta| < \sqrt{s}$ for some $\theta \in \mS^{n-1}$, then one has for any $\phi \in \mS^{n-1}$, $|\frac{1}{\sqrt{d}} \mX^{\top} \phi| < \sqrt{s} + \sqrt{\lambda_{\mathrm{max}}} |\theta - \phi|$. Thus we get by choosing $\phi$ to be in a $s$- net $\mathcal{N}_{s}$ :
$$\P(\lambda_{\mathrm{min}} < s) \leq \left( \frac{3}{s} \right)^n \sup_{\phi \in \mathcal{N}_{s}}\P(|\mX^{\top} \phi| < 2 \sqrt{s d}) + \P(\lambda_{\mathrm{max}} > 1/s) .$$
We now use the Paouris small ball probability bound \cite[Theorem 2]{guedon}, (see also \cite{Pao12}) which states that for an isotropic centered log-concave random variable $Y \in \R^d$, and any $\epsilon \in (0,1/10)$, one has,
$$\P(|Y| \leq \epsilon \sqrt{d}) \leq (c \epsilon)^{\sqrt{d}} ,$$ for some universal constant $c>0.$
As $\mX^{\top} \phi$ is an isotropic, centered, log-concave random variable, we obtain for $d \geq C n^2,$
$$\P(\lambda_{\mathrm{min}} < s) \leq (c s)^{C \sqrt{d}} + \exp(-C/\sqrt{s}) .$$
Finally plugging this back in \eqref{eq:integralsmallball} and choosing $\xi$ to be a suitable negative power of $d,$ we obtain for $d \geq C n^2$,
$$\E \left( -\mathrm{logdet} \big(\frac{1}{d} \mX \mX^{\top}\big) \ds1\{\lambda_{\mathrm{min}}< 1/2\} \right) \leq  n \exp(- d^{1/20}),$$
and thus together with \eqref{eq:firstpart} it yields \eqref{eq:finalsmallball}.
\end{proof}

\section{Proof of \eqref{eq:remarkable}}\label{finalsec}
Recall that $Z \in \R^d$ is a random variable with a twice continuously differentiable density $w$ such that $\int \frac{|\nabla w|^2}{w} < \infty$ and $\int \Vert \nabla^2 w \Vert < \infty$, $h$ is the density of $A Z \in \R^n$ (with $A A^{\top} = \mI_n$), and also we fix $e\in \R^n$ and a sufficiently smooth map\footnote{For instance it is enough that $p$ is twice continuously differentiable, and that the coordinate functions $p_i$ and their derivatives $\frac{\partial p_i}{\partial x_i}$, $\frac{\partial p_i}{\partial x_j}$, $\frac{\partial^2 p_i}{\partial x_i \partial x_j}$ are bounded.} $p : \R^d \rightarrow \R^d$ with $A p(x) = e, \forall x \in \R^d$. We want to prove:
\begin{equation} \label{eq:remarkablerepeat}
e^{\top} I(h) e \leq \int_{\R^d} \bigg( \mathrm{Tr}(Dp^2) + p^{\top} \nabla^2 (- \log w) p \bigg) w \ .
\end{equation}
First we rewrite the right hand side in \eqref{eq:remarkablerepeat} as follows:
$$\int_{\R^d} \bigg( \mathrm{Tr}(Dp^2) + p^{\top} \nabla^2 (- \log w) p \bigg) w  = \int_{\R^d} \frac{(\nabla \cdot (pw))^2}{w} \ .$$
The above identity is a straightforward calculation (with several applications of the one-dimensional integration by parts, which are justified by the assumptions on $p$ and $w$), see \cite{BBN03} for more details. Now we rewrite the left hand side of \eqref{eq:remarkablerepeat}. Using the notation $g_x$ for the partial derivative of a function $g$ in the direction $x$, we have
$$e^{\top} I(h) e = \int_{\R^n} \frac{h_e^2}{h} \ .$$
Next observe that for any $x \in \R^n$ one can write $h(x) = \int_{E^{\perp}} w(A^{\top} x + \cdot)$ where $E \subset \R^d$ is the $n$-dimensional subspace generated by the orthonormal rows of $A$, and thus thanks to the assumptions on $w$ one has:
$$h_e(x) = \int_{A^{\top} x + E^{\perp}} w_{A^{\top} e} = \int_{A^{\top} x + E^{\perp}} \nabla \cdot \left((A^{\top} e) w\right) \ .$$
The key step is now to remark that the condition $\forall x, A p(x) = e$ exactly means that the projection of $p$ on $E$ is $A^{\top} e$, and thus by the Divergence Theorem one has 
$$\int_{A^{\top} x + E^{\perp}} \nabla \cdot \left((A^{\top} e) w\right) = \int_{A^{\top} x + E^{\perp}} \nabla \cdot (p w) \ .$$ 
The proof is concluded with a simple Cauchy-Schwarz inequality:
$$e^{\top} I(h) e = \int_{\R^n} \frac{\left(\int_{A^{\top} x + E^{\perp}} \nabla \cdot (p w) \right)^2}{\int_{A^{\top} x + E^{\perp}} w} \leq \int_{\R^n} \int_{A^{\top} x + E^{\perp}} \frac{(\nabla \cdot (p w))^2}{w} = \int_{\R^d} \frac{(\nabla \cdot (p w))^2}{w} \ .$$

\section{Open problems} \label{sec:open}
This work leaves many questions open. A basic question is whether one could get away with less independence assumption on the matrix $\mX$. Indeed several of the estimates in Section \ref{sec:highdimentropicCLT} and Section \ref{sec:smallball} would work under the assumption that the rows (or the columns) of $\mX$ are i.i.d. from a log-concave distribution in $\R^d$ (or $\R^n$). However it seems that the core of the proof, namely the induction argument from Section \ref{sec:induction}, breaks without the independence assumption for the entries of $\mX$. Thus it remains open whether Theorem \ref{th:mainresult} is true with only row (or column) independence for $\mX$. We note that the case of row independence is probably much harder than column independence.
\newline

As we observed in Section \ref{sec:related}, a natural alternative route to prove Theorem \ref{th:mainresult} (or possibly a variant of it with a different metric) would be to use Stein's method. A straightforward application of existing results yield the suboptimal dimension dependency $d \gg n^6$ for convergence, and it is an intriguing open problem whether the optimal rate $d \gg n^3$ can be obtained with Stein's method. 
\newline

In this paper we consider Wishart matrices with zeroed out diagonal elements in order to avoid further technical difficulties (also for many applications -such as the random geometric graph example- the diagonal elements do not contain relevant information). We believe that Theorem \ref{th:mainresult} remains true with the diagonal included (given an appropriate modification of the Gaussian ensemble). The main difficult is that in the chain rule argument one will have to deal with the law of the diagonal elements conditionally on the other entries. We leave this to further works, but we note that when $\mu$ is the standard Gaussian it is easy to conclude the calculations with these conditional laws.
\newline

In \cite{Eld15} it is proven that when $\mu$ is a standard Gaussian and $d/n \rightarrow +\infty$, one has $\mathrm{TV}(\cW_{n,d}(\mu), \cW_{n,d+1}(\mu)) \rightarrow 0$. It seems conceivable that the techniques develop in this paper could be useful to prove such a result for a more general class of distributions $\mu$. However a major obstacle is that the tools from Section \ref{sec:highdimentropicCLT} are strongly tied to measuring the relative entropy with respect to a standard Gaussian (because it maximizes the entropy), and it is not clear at all how to adapt this part of the proof.
\newline

Finally one may be interested in understanding CLT of the form \eqref{eq:mainresult1} for higher-order interactions. More precisely recall that by denoting $\mX_i$ for the $i^{th}$ column of $\mX$ one can write $\mX \mX^{\top} = \sum_{i=1}^d \mX_i \otimes \mX_i = \sum_{i=1}^d \mX_i^{\otimes 2}$. For $p \in \N$ we may now consider the distribution $\cW_{n,d}^{(p)}$ of  $\frac{1}{\sqrt{d}} \sum_{i=1}^d \mX_i^{\otimes p}$ (for sake of consistency we should remove the non-principal terms in this tensor). The measure $\cW_{n,d}^{(p)}$ have recently gained interest in the machine learning community, see \cite{AGHKT14}. It would be interesting to see if the method described in this paper can be used to understand how large  $d$ needs to be as a function of $n$ and $p$ so that $\cW_{n,d}^{(p)}$ is close to being a Gaussian distribution.

\section*{Acknowledgements}
We are grateful to the anonymous referees for the suggestions that helped improve the paper.
This work was completed while S.G. was an intern at Microsoft Research in Redmond. He thanks the Theory group for its hospitality.

\bibliographystyle{plainnat}
\bibliography{convexbib}
\end{document}